\documentclass[12pt]{amsart}
\usepackage{amssymb,latexsym,amsmath,amsthm,amscd}
\usepackage{setspace}
\usepackage{verbatim}
\usepackage[active]{srcltx}
\usepackage[all]{xy} \xyoption{arc}
\usepackage[left=3cm,top=2cm,right=3cm,bottom = 2cm]{geometry}
\usepackage{graphicx}
\usepackage[usenames,dvipsnames]{color}
\usepackage{charter}

\theoremstyle{plain}
\newtheorem{thm}{Theorem}[section]
\newtheorem{lem}[thm]{Lemma}
\newtheorem{prop}[thm]{Proposition}

\theoremstyle{definition}
\newtheorem{dfn}[thm]{Definition}
\newtheorem{ex}[thm]{Example}

\theoremstyle{remark}
\newtheorem{rmk}[thm]{Remark}

\newcommand{\cO}{\mathcal{O}}

\newcommand{\cS}{\mathcal{S}}

\newcommand{\cV}{\mathcal{V}}

\DeclareMathOperator{\wt}{wt}

\DeclareMathOperator{\Tr}{Tr}

\DeclareMathOperator{\SL}{SL}
\DeclareMathOperator{\PSL}{PSL}

\newcommand*{\df}{\mathrel{\vcenter{\baselineskip0.5ex \lineskiplimit0pt
                     \hbox{\scriptsize.}\hbox{\scriptsize.}}} =}

\providecommand{\abs}[1]{\left\lvert#1\right\rvert}
\providecommand{\twomat}[4]{\left(\begin{matrix}#1&#2\\#3&#4\end{matrix}\right)}

\providecommand{\twovec}[2]{\left(\begin{matrix}#1\\#2\end{matrix}\right)}

\newcommand{\QQ}{\mathbf{Q}}
\newcommand{\FF}{\mathbf{F}}
\newcommand{\CC}{\mathbf{C}}

\newcommand{\ZZ}{\mathbf{Z}}

\newcommand{\NN}{\mathbf{N}}

\newcommand{\cVbar}{\overline{\cV}}

\newcommand{\cSbar}{\overline{\cS}}

\begin{document}
%Title
\title{On the structure of modules of vector valued modular forms}
\author{Cameron Franc and Geoffrey Mason}
\date{}

\begin{abstract}
If $\rho$ denotes a finite dimensional complex representation of $\SL_2(\ZZ)$, then it is known that the module $M(\rho)$ of vector valued modular forms for $\rho$ is free and of finite rank over the ring $M$ of scalar modular forms of level one. This paper initiates a general study of the structure of $M(\rho)$. Among our results are absolute upper and lower bounds, depending only on the dimension of $\rho$, on the weights of generators for $M(\rho)$, as well as upper bounds on the multiplicities of weights of generators of $M(\rho)$. We provide evidence, both computational and theoretical, that a stronger three-term multiplicity bound might hold. An important step in establishing the multiplicity bounds is to show that there exists a free-basis for $M(\rho)$ in which the matrix of the modular derivative operator does not contain any copies of the Eisenstein series $E_6$ of weight six.
\end{abstract}
\maketitle

\setcounter{tocdepth}{1}
\tableofcontents

\setlength{\parskip}{2mm}

\section{Introduction}
If $\rho$ is a finite-dimensional complex representation of $\SL_2(\ZZ)$ of dimension $d$, then the module $M(\rho)$ of vector valued modular forms for $\rho$ is known \cite{MM}, \cite{CandeloriFranc} to be free of rank $d$ over the ring $M$ of classical scalar modular forms of level one.\ A basic problem about $M(\rho)$ is then to determine the weights of a generating set of modular forms in this module.\ These are invariants of the isomorphism class of $\rho$.\ In \cite{CandeloriFranc} it was observed that this question is tantamount to determining the decomposition of a certain vector bundle $\cVbar_0(\rho)$ on the moduli stack of elliptic curves into line bundles.\ In dimension less than six, some results on this questions have been obtained by Marks \cite{Marks}, but otherwise very little has been proved about the general situation. This paper initiates a general study of this question.

We begin in Section \ref{s:weightprofiles} by introducing the \emph{weight profile} of $\rho$, which is the tuple $(k_1,\ldots, k_d)$ of weights of generators for $M(\rho)$, ordered so that $k_i \leq k_{i+1}$ for all $i$. In \cite{Mason} one finds a proof that $1-d \leq k_1$ for irreducible representations $\rho$. Using a slight generalization of this argument, combined with Serre duality, we show (Lemma \ref{l:weightprofilebounds}) that there is an upper bound  $k_d \leq d+10$ as well.\ In particular, there are only finitely many weight profiles for irreducible representations in each dimension.\ Section \ref{s:weightprofiles} explains how Proposition \ref{p:dimension} and Westbury's description of the character variety of $\SL_2(\ZZ)$ (recalled as Theorem \ref{t:westbury} below) can be used to enumerate a finite list containing all weight profiles of irreducible representations of $\SL_2(\ZZ)$ of fixed dimension (and possibly some weight profiles that do not occur in practice) that is considerably shorter than the finite list provided by the weight bounds $1-d \leq k_1 \leq k_d \leq d+10$. Section \ref{s:examples} contains the results of some of these computations.

Sections \ref{s:Rmodule} and \ref{s:multiplicitybounds} use the differential structure of $M(\rho)$ afforded by the modular derivative to deduce further information about the weight profile of $\rho$.\ The most notable new results are the no-gap Lemma (Lemma \ref{lemmanogap}) and the weight multiplicity bound of Theorem \ref{t:multiplicitybounds}.\ The no-gap Lemma states that if $\rho$ is irreducible, then no gap larger than two occurs in its weight profile.\ If $m_1,\ldots, m_r$ denote the multiplicities of the distinct weights in the weight profile of an irreducible representation $\rho$, then the weight multiplicity bound states that
\[
  m_j \leq \sum_{t \geq 0} m_{j+1-2t}\ \mbox{and}\   m_j \leq \sum_{t \geq 0}m_{j-1+2t},
\]
where $m_i = 0$ if $i < 1$ or $i > r$.\ Since the tuple $\pi_\rho = (m_1,\ldots, m_r)$ is an ordered partition of $d$, this implies in particular that $m_j \leq d/2$ for all $j$.\ The proof of Theorem \ref{t:multiplicitybounds} seems to be new and of considerable interest. It uses the fact, proved in Theorem \ref{l:goodbasis}, that there exists a choice of basis for $M(\rho)$ in which the matrix of the modular derivative does not contain any copies of the weight $6$ Eisenstein series $E_6$.

Aside from the intrinsic interest of the weight profiles of representations of $\SL_2(\ZZ)$, there would be practical benefits to understanding them better.\ For example, if $\rho$ denotes the permutation representation of $\SL_2(\ZZ)$ acting on the cosets of some finite index subgroup $\Gamma$, then $M_k(\rho) \cong M_k(\Gamma)$, where $M_k(\Gamma)$ denotes the space of scalar modular forms for $\Gamma$ of weight $k$.\ If $\rho$ decomposes into irreducible representations $\rho \cong \bigoplus_i \rho_i$, then one obtains a corresponding decomposition $M_k(\Gamma) \cong \bigoplus_i M_k(\rho_i)$.\ Thus, for example,
\begin{equation}
\label{eq:wt1}
  \dim M_1(\Gamma) = \sum_i \dim \rho_i\cdot m_1(\rho_i),
\end{equation}
where $m_1(\rho_i)$ denotes the multiplicity of the weight $1$ in the weight profile of $\rho_i$ (here we've used the fact that the weight profile of a representation of finite image consists of positive integers).

We explored the idea of using the decomposition (\ref{eq:wt1}) and the results of \cite{CandeloriFranc} to study the dimensions of spaces of modular forms of weight one on $\Gamma(p)$ for a prime $p$.\ We were pleased to observe that when $p \equiv 3 \pmod{4}$, and if $\rho_1$ and $\rho_2$ denote the irreducible representations of $\SL_2(\FF_p)$ obtained as certain constituents of reducible principal series representations, then the Euler characteristics of the corresponding vector bundles $\cVbar_0(\rho_i)$ equal $\frac 12(1 \pm h(-p))$, where $h(-p)$ denotes the class number of $\QQ(\sqrt{-p})$. Using this, it is not too hard to show that $\dim M_1(\rho_i) \geq \frac 12(1 + h(-p))$. This elementary argument detects the dihedral theta series of weight one without writing them down explicitly, and is presumably well-known to experts\footnote{Nevertheless, it is worth remarking that the term $h(-p)$ in the Euler characteristic arises via the exponents of $\rho_i(T)$ through Dirichlet's analytic class number formula.}. Unfortunately, the elementary arguments that \cite{CandeloriFranc} enables do not, by themselves, shed any new light on the question of dimensions of spaces of modular forms of weight one.\ However, the question of the module structure of $M(\rho)$ seems to be a richer one, and a methodical study of $M(\rho)$ for general $\rho$ could conceivably lead to a better understanding of scalar forms that have so far resisted available techniques.\ This is part of the impetus that drove this work. 

Let us conclude the introduction by describing our notation and conventions.

In this note $\rho$ will always denote a finite-dimensional complex representation of $\SL_2(\ZZ)$, usually irreducible.\ It will often be convenient to assume that $\rho(S^2)$ is a scalar.\ Then necessarily $\rho(S^2) = \pm I$. If $\rho(S) = I$ then $\rho$ is said to be \emph{even}, while if $\rho(S) = -I$ then $\rho$ is said to be \emph{odd}.\ If $\rho$ is even or odd, then the weights of nonzero vector valued modular forms for $\rho$ must have the same parity as $\rho$.\ Note that all irreducible representations of $\SL_2(\ZZ)$ are either even or odd. 
The notation $\rho^\vee$ denotes the dual representation of $\rho$. Let
\[
  T = \twomat 1101, \quad\quad S = \twomat{0}{-1}{1}{0}, \quad\quad R = ST = \twomat{0}{-1}{1}{1}.
\]
Let $\chi$ denote the character of $\eta^2$, so that $\chi(T) = e^{2\pi i/12}$. Write $\xi = e^{2\pi i/6}$. If $L$ is a matrix such that $\rho(T) = e^{2\pi i L}$, then we call $L$ a \emph{choice of exponents} for $\rho$. Recall that since $\det e^M = e^{\Tr(M)}$, the quantity $12\Tr(L)$ is an integer for any choice of exponents $L$ for $\rho$. Let $\cVbar_{k,L}(\rho)$ denote the vector bundle introduced in \cite{CandeloriFranc}. If $L$ has eigenvalues with real part in $[0,1)$ then we write simply $\cVbar_{k}(\rho)$ for $\cVbar_{k,L}(\rho)$. Similarly, if the eigenvalues have real part in $(0,1]$ then we write $\cSbar_k(\rho)$ for $\cVbar_{k,L}(\rho)$. The global sections of $\cVbar_k(\rho)$ and $\cSbar_k(\rho)$ are the spaces of weight $k$ vector valued holomorphic modular forms $M_k(\rho)$ and cusp forms $S_k(\rho)$, respectively, for $\rho$. Note that if $L$ and $L_0$ denote choices of exponents for $\rho(T)$ adapted to $[0,1)$ and $(1,0]$, respectively, then $\Tr(L_0) = \Tr(L) + m$ where $m$ is the multiplicity of one as an eigenvalue of $\rho(T)$.

\section{Weight profiles}
\label{s:weightprofiles}
In \cite{CandeloriFranc}, the Euler characterstic of the bundles $\cVbar_k(\rho)$ was computed. When $k$ is large enough, depending on $\rho$, the Euler characteristic agrees with the dimension of $M_k(\rho)$. The following proposition makes this precise.
\begin{prop}
\label{p:dimension}
Let $\rho$ denote an irreducible representation of dimension $d$, let $L$ denote a standard choice of exponents for $\rho(T)$, and let $L_0$ denote a cuspidal choice of exponents for $\rho(T)$. Then one has
\[
  \dim M_k(\rho) = \begin{cases}
0 & k < \frac{12}{d}\Tr(L)+1-d,\\
\chi(\cVbar_k(\rho)) + \dim S_{2-k}(\rho^{\vee}) & \frac{12}{d}\Tr(L)+1-d \leq k \leq \frac{12}{d}\Tr(L) + d-11,\\
\chi(\cVbar_k(\rho)) & k > \frac{12}{d}\Tr(L) + d-11,
  \end{cases}
\]
and
\[
  \dim S_k(\rho) = \begin{cases}
0 & k < \frac{12}{d}\Tr(L_0)+1-d,\\
\chi(\cSbar_k(\rho)) + \dim M_{2-k}(\rho^{\vee}) & \frac{12}{d}\Tr(L_0)+1-d \leq k \leq \frac{12}{d}\Tr(L_0) + d-11,\\
\chi(\cSbar_k(\rho)) & k > \frac{12}{d}\Tr(L_0) + d-11.
  \end{cases}
\] 
\end{prop}
\begin{proof}
Recall (Proposition 3.14 of \cite{CandeloriFranc}) that if $\rho$ is irreducible of dimension $d$, $L$ is a choice of exponents for $\rho$, and $k$ is the minimal integer such that $h^0(\cVbar_{k,L}(\rho)) \neq 0$, then
\[
  k \geq \frac{12}{d}\Tr(L) + 1-d.
\]
Let $m$ be the multiplicity of one as an eigenvalue for $\rho(T)$, let $k$ be the minimal weight for $\rho$, and let $\ell$ be the minimal integer such that $S_\ell(\rho) \neq 0$. Then if $L$ is a standard choice of exponents for $\rho$,
\[
  \ell \geq \frac{12\Tr(L)+12m}{d} + 1-d.
\]
However, to apply Serre duality to the computation of dimensions of spaces of modular forms, one wishes to know when $S_{2-k}(\rho^\vee)$ is nonzero. Note that the multiplicity of one as an eigenvalue of $\rho^\vee(T)$ is also $m$. If $L$ is a standard choice of exponents for $\rho$, and if $L^\vee$ is a standard choice of exponents for $\rho^\vee$, then $\Tr(L^\vee) = d-\Tr(L)-m$. We thus see that if $S_{2-k}(\rho^\vee) \neq 0$ then
\[
  k \leq \frac{12}{d}\Tr(L)+d-11
\]
This proves the claim about $M_k(\rho)$. The proof of the claim for $S_k(\rho)$ is similar.
\end{proof}

The middle cases of Proposition \ref{p:dimension} comprise at most $2d-11$ weights. Half of these can be eliminated using parity considerations, but in general the other half might be difficult to compute. When $d \leq 5$, however, Proposition \ref{p:dimension} gives explicit formulae for $\dim M_k(\rho)$ and $\dim S_k(\rho)$ in all weights. For general $d$ one can use Proposition \ref{p:dimension} and positivity to narrow down the possibilities for $\dim M_k(\rho)$ and $\dim S_k(\rho)$ in low weights to a finite number of possibilities -- see Theorem \ref{t:algorithm} below.

The free module theorem for vector valued modular forms states that the module $M(\rho)$ of vector valued modular forms for $\rho$ is free of rank $d = \dim \rho$ over the ring $M$ of scalar modular forms of level one.\ This result follows, for example, from the complete decomposability of vector bundles on the moduli stack of elliptic curves \cite{CandeloriFranc}.\ The free-module theorem also holds for the module $S(\rho)$ of cusp forms for $\rho$, and more generally for the module $M_L(\rho)$ of modular forms for $\rho$ relative to any given choice of exponents $L$ for $\rho$. 

Let us write
\[
  \cVbar_k(\rho) = \bigoplus_{j = 1}^d \cO(k-k_i),\quad\quad \cSbar_k(\rho) = \bigoplus_{j = 1}^d \cO(k-\ell_i)
\]
for integers $k_i$, $\ell_i$ with $k_j \leq k_{j+1}$ and  $\ell_j \leq \ell_{j+1}$ for all $j$. The integers $-k_i$ are the \emph{roots} of $\rho$. The tuples $(k_1,\ldots, k_d)$ and $(\ell_1,\ldots, \ell_d)$ are called the \emph{weight profile} and \emph{cuspidal weight profile} of $\rho$, respectively. More generally, if $\cVbar_{k,L}(\rho) = \bigoplus_{j = 1}^d \cO(k-k_i)$ then we call $(k_1,\ldots, k_d)$ the \emph{$L$-adapted weight profile} of $\rho$. If $(k_j)$ denotes a weight profile, then the \emph{type profile}, or more simply, the \emph{type} of $\rho$ is the tuple $(0,k_2-k_1,\ldots, k_d-k_1)$. Obviously one can recover the weight profile from the knowledge of the type and the minimal weight. Conversely,
\begin{lem}
\label{l:weightprofilesum}
  Let $\rho$ denote a representation with weight profile $(k_1,\ldots, k_d)$ and standard choice of exponents $L$. Then $\sum_{j = 1}^d k_j = 12\Tr(L)$. In particular, the minimal weight of $\rho$ is determined by the type of $\rho$ and $\Tr(L)$ for a standard choice of exponents $L$ for $\rho$.
\end{lem}
\begin{proof}
Proposition 3.6 of \cite{CandeloriFranc} observes that $\det \cVbar_k(\rho) \cong \cO(dk-12\Tr(L))$, and the first claim follows from this. Thus $k_1 = (12\Tr(L) - \sum_{j = 1}^d(k_j-k_1))/d$, and this proves the second claim.
\end{proof}

\begin{rmk}
\label{r:duality}
Recall from Proposition 3.13 of \cite{CandeloriFranc} that $\cVbar_k(\rho)^\vee \cong \cSbar_{12-k}(\rho^\vee)$. It follows that if $(k_1,\ldots, k_d)$ is the weight profile of $\rho$, then $(12-k_d, \ldots, 12-k_1)$ is the cuspidal weight profile of $\rho^\vee$. Hence if $\rho$ is such that $1$ is not an eigenvalue of $\rho(T)$, so that $\cSbar_{12-k}(\rho^\vee)=\cVbar_{12-k}(\rho^\vee)$, the dual weight profile of $\rho$ is $(12-k_d,\ldots, 12-k_1)$. If moreover $\rho$ is self dual, this implies that $k_j +k_{d+1-j}=12$ for all $j$. Hence Lemma \ref{l:weightprofilesum} implies that $\Tr(L)=d/2$ for such representations $\rho$.  The problem of relating the weight profile of a representation with that of its dual is an interesting and likely tractable open problem.
\end{rmk}

\begin{lem}
\label{l:weightprofilebounds}
Let $\rho$ be an irreducible representation, and let $(k_1,\ldots, k_d)$ be its weight profile. Then
\[
  \frac{12}{d}\Tr(L) + 1-d \leq k_1 \leq k_d \leq \frac{12}{d}\Tr(L) + d-1.
\]
In particular, for all irreducible representations of dimension $d$, the weight profiles lie in the range $[1-d, d+10]$. There are thus finitely many weight profiles for irreducible representations in each dimension.
\end{lem}
\begin{proof}
The lower bound on $k_1$ is well-known (see e.g. Proposition 3.14 of \cite{CandeloriFranc}). By Remark \ref{r:duality}, the cuspidal weight profile of $\rho^\vee$ is $(12-k_d,\ldots, 12-k_1)$. Let $L^\vee$ be a choice of exponents for $\rho^\vee$ adapted to the interval $(0,1]$, so that $\cSbar_{k}(\rho^\vee) = \cVbar_{k,L^\vee}(\rho^\vee)$. Then by the slight generalization of the Wronskian argument given in Proposition 3.14 of \cite{CandeloriFranc},
\[
  \frac{12}{d}\Tr(L^\vee) +1-d \leq 12-k_d,
\]
and thus $k_d \leq 11+d - \frac{12}{d}\Tr(L^\vee)$. Note that $\Tr(L^\vee) = d-\Tr(L)$ since $L^\vee$ denotes the cuspidal exponents for $\rho^\vee$. Thus $k_d \leq \frac{12}{d}\Tr(L)+d-1$.
\end{proof}

\begin{rmk}
\label{r:bestpossiblebounds}
One can show that the bounds of Lemma \ref{l:weightprofilebounds} are sharp.
\end{rmk}

Tuba and Wenzl \cite{TubaWenzl} have described all irreducible representations of $\SL_2(\ZZ)$ in dimension less than six. One can use this and Proposition \ref{p:dimension} to compute all types and minimal weights in dimension less than six. The results are below. These computations are consistent with, and add precision to, the computations in \cite{Marks}.
\begin{ex}
\label{exrk4}
We list the possible types of irreducible representations in dimension $\leq 4$, along with the minimal weight. 
\begin{align*}
\begin{array}{c|l|l}
\textrm{Dimension} & \textrm{Type} & k_1\\
\hline
1 &(0)&12\Tr(L)\\
2 &(0,2)&6\Tr(L)-1\\
3 &(0,2,4)&4\Tr(L)-2\\
4 &(0,2,4,6)&3\Tr(L)-3\\
4 &(0,2,2,4)&3\Tr(L)-2
\end{array}
\end{align*}
\end{ex}

\begin{rmk}
In \cite{TubaWenzl} is it shown that up to a choice of square root of $\det(T)$, the eigenvalues of $\rho(T)$ determine four dimensional irreducible representations of $\SL_2(\ZZ)$. The two possibilities for the type in dimension $4$ correspond to the two choices of square root.
\end{rmk}

\begin{ex}
\label{exrk5}
The case of five dimensional irreducible representations is more interesting. In this case it need not be true that $5 \mid \Tr(L)$. Using \cite{TubaWenzl}, one can compute the minimal weights and types. To express the result it is best to write $\Tr(L) = \frac{a}{12}$ where $0 \leq a \leq 59$. One finds the following possibilities:
\[
\begin{array}{c|c|c}
a \pmod{5} &\textrm{Type} & k_1\\
\hline
0 &(0,2,4,6,8)&\left(a-20\right)/5\\
1 &(0,2,4,4,6)&\left(a-16\right)/5\\
2 &(0,2,2,4,4)&\left(a-12\right)/5\\
3 &(0,0,2,2,4)&\left(a-8\right)/5\\
4 &(0,2,2,4,6)&\left(a-14\right)/5
\end{array}
\]
\end{ex}

Our next goal is to describe an algorithm for enumerating a list that contains all possible types of irreducible representations in a given dimension.
\begin{thm}
\label{t:algorithm}
  There exists an algorithm that takes as input an integer $d \geq 1$ and the resulting output is a finite list of $d$-tuples of positive integers that contains all possible types of irreducible representations of $\SL_2(\ZZ)$ of dimension $d$.
\end{thm}
\begin{proof}
By the no-gap lemma (Lemma \ref{lemmanogap} below), one could simply enumerate all possible sequences of integers $(x_1,\ldots, x_d)$ where $x_1 = 0$ and such that $0 \leq x_{i+1}-x_i \leq 1$ for $i = 1,\ldots, d-1$. There are $2^{d-1}$ such sequences.
\end{proof}

A large number of the $2^{d-1}$ possible types given by the no-gap lemma do not occur in practice. A number of additional restrictions on the types, arising from the differential structure on $M(\rho)$, are described in Sections \ref{s:Rmodule} and \ref{s:multiplicitybounds} below. Proposition \ref{p:dimension} and Westbury's description \cite{Westbury} of the irreducible components of the character variety of semistable representations of $\SL_2(\ZZ)$ can also be used to cut down the possibilities dramatically. We describe this next.

\begin{thm}
\label{t:westbury}
The character variety $X_d$ classifying $d$-dimensional semistable representations of $\PSL_2(\ZZ)$ is an affine algebraic variety that decomposes into a disjoint union of irreducible components $X_d = \amalg_\alpha X_\alpha$ indexed by tuples of nonnegative integers $\alpha = (a,b;x,y,z)$ satisfying $a+b = x+y+z = d$. A given irreducible representation $\rho$ of $\PSL_2(\ZZ)$ of dimension $d$ lies on the component $X_\alpha$ indexed by $\alpha = (a,b;x,y,z)$ where $a$ and $b$ are the multiplicities of $1$ and $-1$, respectively, as eigenvalues of $\rho(S)$, and where $x$, $y$ and $z$ denote the multiplicities of $1$, $\zeta = e^{\frac{2\pi i}{3}}$ and $\zeta^2$, respectively, as eigenvalues of $\rho(R)$.
\end{thm}
\begin{proof}
This result was originally proved by Bruce Westbury \cite{Westbury}, but it remains unpublished. See Section 2 of \cite{LeBruyn} for more information on the character variety of the modular group.
\end{proof}
\begin{rmk}
\label{r:sl2}
If $\rho$ is an odd irreducible representation of $\SL_2(\ZZ)$, then $\rho \otimes \chi$ is an irreducible representation of $\PSL_2(\ZZ)$. Thus, Theorem \ref{t:westbury} allows one to give a similar description for the character variety of $\SL_2(\ZZ)$.
\end{rmk}

Fix an irreducible representation $\rho$ of dimension $d$, let $L$ denote a standard choice of exponents for $\rho(T)$, so that $\Tr(L) \in [0,d)$, and write $s = \Tr(\rho(S))$, $r_1 = \Tr(\rho(R))$ and $r_2 = \Tr(\rho(R^2))$. The quantities $s$, $r_1$ and $r_2$ are constant on each irreducible component of the character variety by Theorem \ref{t:westbury} and Remark \ref{r:sl2}. Similarly, $\det \rho$ is constant on components of the character variety, so that $\Tr(L)$ takes on at most $d$ values across each component of the character variety. Thus, if we perform the following computation for fixed $s$, $r_1$, $r_2$ and for the $d$ possible values of $\Tr(L)$ for representations on the component of the character variety that contains $\rho$, then the result is a finite computation that gives all possible types of representations on the irreducible component that contains $\rho$. Thus, we need only describe how to narrow down the possibilities for the type of our fixed $\rho$ to a finite list.

In order to describe the computation, let $\ell_1,\ldots \ell_r$ denote the increasing sequence of integers in the interval between $(12/d)\Tr(L) + 1-d$ and $(12/d)\Tr(L) +d-11$ with the same parity as $\rho$, and set $a_j = \dim S_{2-\ell_j}(\rho^\vee)$ for each $j$. Then by Proposition \ref{p:dimension},
\begin{equation}
\label{eq:poincareseries}
\frac{\sum_{j = 1}^d T^{k_j}}{(1-T^4)(1-T^6)} = \sum_{k \geq \ell_1}\chi(\cVbar_k(\rho))T^k + \sum_{j=1}^r a_jT^{\ell_j}.
\end{equation}
By Corollary 6.2 of \cite{CandeloriFranc}, 
\begin{align*}
\sum_{k \geq \ell_1} \chi(\cVbar_k(\rho))T^k=&T^{\ell_1}\left(\frac{5d-12\Tr(L)}{12}\frac{1}{1-T^2}+\frac{s}{4}\frac{i^{\ell_1}}{1+T^2}+\frac{r_1}{3(1-\zeta)}\frac{\xi^{\ell_1}}{1-\zeta T^2}+\right.\\
&\quad \left.\frac{r_2}{3(1-\zeta^2)}\frac{\zeta^{\ell_1}}{1-\zeta^2 T^2}+\frac{d}{12}\frac{\ell_1 -(\ell_1-2)T^{2}}{(1-T^2)^2}\right)
\end{align*}
where $\xi = e^{2\pi i/6}$ and $\zeta = \xi^2$. It follows that equation (\ref{eq:poincareseries}) yields an explicit and computable equation of the form 
\begin{equation}
\label{eq:weightprofilepoly}
\sum_{j = 1}^d T^{k_j} = T^{\ell_1}P(T) + \sum_{j = 1}^ra_jT^{\ell_j}(1-T^4)(1-T^6)
\end{equation}
where $P(T)$ is a polynomial of degree at most $8$ with integer coefficients. Note that $P(T)$ only depends on $d$, $s = \Tr(\rho(S))$, $r_1 = \Tr(\rho(R))$, $r_2 = \Tr(\rho(R^2))$ and $\Tr(L)$ (since $\ell_1$ was defined in terms of $\Tr(L)$ and $d$ via Proposition \ref{p:dimension}). 
\begin{lem}
\label{l:finiteness}
There are only finitely many solutions to equation (\ref{eq:weightprofilepoly}) in nonnegative integers $a_j$.
\end{lem}
\begin{proof}
By comparison with the left side of the equation, the coefficients of the right hand side of equation (\ref{eq:weightprofilepoly}) must be nonnegative integers that are no larger than $d$. Write $P(T) = \sum_{j = 0}^{8}b_jT^j$. The coefficient of $T^{\ell_1}$ in (\ref{eq:weightprofilepoly}) is of the form $p_0 + a_1$. Thus $0 \leq a_1 \leq d-p_0$, so that there are finitely many possibilities for $a_1$. Similarly, the coefficient of $T^{\ell_j}$ is of the form $p_k + a_j + Q(a_1,\ldots, a_{j-1})$ for some polynomial $Q(a_1,\ldots, a_{j-1})$. By induction, this polynomial takes on finitely many values, and so if $M$ is the maximal such value, we find that $0 \leq a_j \leq d-p_k-M$, for some index $k$. This proves the lemma.
\end{proof}

By Lemma \ref{l:finiteness} it is thus possible to enumerate the finitely many solutions to equation (\ref{eq:weightprofilepoly}) in nonnegative integers, and thereby find all types of irreducible representations $\rho$ of dimension $d$ and with fixed values of $\Tr(\rho(S))$, $\Tr(\rho(R))$, $\Tr(\rho(R^2))$, and $\Tr(L)$. By Theorem \ref{t:westbury}, this allows one to describe a finite list of all types of irreducible representations in dimension $d$. We have implemented these computations using \emph{Sage}, and we were able to run the algorithm in dimensions up to and including twelve before the computations began to run into memory limitations. Some of these results are listed in Section \ref{s:examples}. The number of types that are output by this algorithm tends to be exponential in $d$, but it is a much smaller number than the $2^{d-1}$ given by the no-gap lemma. Nevertheless, there are many types that arise in this way that do not actually occur. Some of these possibilities can be eliminated using the results of Section \ref{s:Rmodule} below, but in general it seems to be an open problem to determine exactly what type profiles do occur in each dimension.

We end this section by explaining how to extend this finiteness result to all representations.
\begin{prop}\label{proptypeprofile}  Fix a positive integer $d$. There are only finitely many possible weight profiles for representations of $\SL_2(\ZZ)$ of dimension $d$.
\end{prop}
\begin{proof} 
We have explained the proof of Proposition \ref{proptypeprofile} for irreducible representations. Suppose that 
\[0\rightarrow \rho_1\rightarrow\rho\rightarrow\rho_2\rightarrow 0\]
is a short exact sequence of representations of $\SL_2(\ZZ)$. After applying the functor $M$, there is an exact sequence
\[
0\rightarrow M(\rho_1)\rightarrow M(\rho)\rightarrow M(\rho/\rho_1)\]
(cf.\ \cite{MM} for more details).\
This shows that at the level of multisets, the
set of weights for $\rho$ is contained in the union of the corresponding multisets for
$\rho_1$ and $\rho_2$.\ Taking a composition series for $\rho$, this
shows that the number of weight profiles in dimension $d$ is no more than $d$ times the maximum of the number of weight profiles for an irreducible of dimension no greater than $d$.\ So finiteness 
in general follows from the irreducible case.
\end{proof}

\section{Exploiting the differential structure}
\label{s:Rmodule}
Let $R\df M\langle D\rangle$ be the  algebra of modular differential operators, where $D$ acts on modular forms of weight $k$ via the usual operator
\begin{eqnarray}
\label{Dkact}
  D_k \df q\frac{d}{dq} - \frac{k}{12}E_2.
\end{eqnarray}
Since we have normalized the Eisenstein series $E_4$ and $E_6$ of weights $4$ and $6$ to have constant term equal to $1$, one has $D(E_4) = -\frac{1}{3}E_6$ and $D(E_6) = -\frac{1}{2}E_4^2$.\ Formally, elements of $R$ are polynomials  $\sum_i f_iD^i$ in $D$ with coefficients $f_i\in M$, however $R$ is noncommutative (although associative).\ Multiplication is implemented using the identity $Df = fD+D(f)\ (f\in M)$. If we give $D$ degree $2$ then $R$ is an $\NN$-graded algebra.

One knows (\cite{Mason}, \cite{MM}) that $M(\rho)$ is a $\ZZ$-graded left $R$-module.\ Elements of $M$ act by multiplication
and $D$ acts via the obvious extension of (\ref{Dkact}) to vvmfs of weight $k$.\ It is the exploitation of this fact that
underlies the results in the present Section.\ Actually, the structure of $M(\rho)$ as $R$-module is an interesting topic in its own right,
but we will resist the temptation to axiomatize the situation, and simply record some of the relevant features.

The free module theorem (\cite{MM}, \cite{CandeloriFranc}) says that $M(\rho)$ is a free $M$-module of rank $\dim\rho$.\ On the other hand, $M(\rho)$ is a torsion $R$-module:\ every element in $M(\rho)$ has a nonzero annihilator in $R$.\ We will use the following more precise version of this fact in the case that $\rho$ is irreducible.

\begin{lem}
\label{lemmaDFs}
\label{lemmaRrk} 
Assume that $\rho$ is irreducible of dimension $d$, and let 
$F\in M_k(\rho)$ be nonzero. Then the following hold:
\begin{enumerate}
\item[(a)] $F, DF, \hdots, D^{d-1}F$ are linearly independent over $M$,
\item[(b)] $F, DF, \hdots, D^{d}F$ are linearly dependent over $M$ (that is, some polynomial of degree $d$ annihilates $F$, but none of degree less than $d$),
\item[(c)] If $0\not=N\subseteq M(\rho)$ is a graded $R$-submodule that is free of rank $r$ as an $M$-module, then $r = d$.
\end{enumerate}
\end{lem}
\begin{proof} To say that a nonzero polynomial of degree $n$ in $R$ annihilates $F$ just means (taking the grading into account) that  there is a  relation 
\begin{equation}
\label{eq:rel}
\sum_{i=0}^{n} f_iD^iF=0,
\end{equation}
where each $f_i \in M_{k'-k-2i}$ for some fixed $k'$ and $f_n\not=0$.

Relation (\ref{eq:rel}) tells us that $F$ satisfies a modular linear differential equation, or MLDE (cf.\ \cite{Mason}, \cite{FM1}), of order $n$.\ If there are \emph{no} such relations with $n=d$ then $F$, $DF$, $\hdots$, $D^dF$ are linearly independent over $M$, whence they span a free $M$-submodule of $M(\rho)$ of rank $d+1$.\ Since $M(\rho)$ is free of rank $d$ this is not possible, and this contradiction establishes part (b).

On the other hand, suppose (\ref{eq:rel}) holds with $n\leq d-1$.\ As an order $n$ MLDE, the solution space of (\ref{eq:rel}) is $n$-dimensional, and therefore the span $E$ of the components of $F$ (a subspace of the solution space) has dimension \emph{less} than $d$.\ However, because $\rho$ is irreducible, the components of $F$ span an $\SL_2(\ZZ)$-module that affords a representation equivalent to $\rho$.\ In particular, the span of these components has dimension $d$. This contradiction proves (a).

As for (c), choose a nonzero form $F \in N\cap M_k(\rho)$ for some $k$. By part (a), $F, \hdots, D^{d-1}F$ generate a free $M$-submodule of $N$ of rank $d$, so that $r\geq d$.\ On the other hand, $r\leq d$ because $M(\rho)$ is free of rank $d$.
Thus $r=d$, and the proof of the lemma is complete.
\end{proof}

Lemma \ref{lemmaRrk} implies the no-gap lemma (Lemma \ref{lemmanogap}), Lemma \ref{lemmacyclic} and Proposition \ref{p:cyclic} below.

\begin{dfn} 
Let $\rho$ denote an even or odd representation of $\SL_2(\ZZ)$, so that all weights in the weight profile of $\rho$ have the same parity. A \emph{gap} in the weight profile of $\rho$ is an integer $k$ with the following properties: $k$ has the same parity as the weights of $\rho$; there are weights of $\rho$ which are \emph{less}
than $k$ and weights  which are \emph{greater} than $k$, but \emph{no} weights
equal to $k$.
\end{dfn}
\begin{lem}[No-gap lemma]
\label{lemmanogap} 
Suppose that $\rho$ is an irreducible representation of $\SL_2(\ZZ)$. Then the weight profile of $\rho$ has no gaps.
\end{lem}
\begin{proof} 
Suppose that $k$ is a gap in the weight profile of $\rho$.\ Then we can divide a set $X$ of (homogeneous) generators of $M(\rho)$ into two nonempty subsets $X=X_1\cup{X_2}$ such that all weights of generators in $X_1$ are $\leq k-2$, and all weights of generators in $X_2$ are $\geq k+2$.\ Note that $\abs{X_1}+\abs{X_2}=\dim\rho$.

Let $F\in X_1$.\ Then $\wt(D(F))=\wt(F)+2 \leq k$, so if we write $D(F)$ as an $M$-linear combination of generators in $X$, all of those generators have weight $\leq k$, and hence they lie in $X_1$.\ This shows that the $M$-submodule $M_1\subseteq M(\rho)$ spanned by the generators in $X_1$ is in fact an $R$-submodule.\ Since $X_2$ is nonempty, $M_1$ has $M$-rank $\abs{X_1}$. But $\abs{X_1} < \dim\rho$, and this contradicts Lemma \ref{lemmaRrk}(c).
\end{proof}

\begin{dfn} 
\label{d:multiplicities}
Fix a representation $\rho$ of $\SL_2(\ZZ)$ of dimension $d$. We denote by $\pi_{\rho}$ the ordered partition consisting of the \emph{multiplicities} of the weights that occur in the weight profile of $\rho$. Thus $\pi_{\rho}=(m_1, \ldots, m_r)$ means that the distinct weights that occur are $k_1'< \cdots <k_r'$ and the weight profile is 
\[\pi_\rho = (\underbrace{k_1', \ldots, k_1'}_{m_1}, \underbrace{k_2', \ldots , k_2'}_{m_2},\ldots).\] 
Similarly, we have the \emph{cuspidal} analog $\pi_{\rho}^S$ which records the multiplicities of the generating weights in the cuspidal weight profile of $\rho$.
\end{dfn}

\begin{rmk}
\label{r:multiplicities} If $\rho$ is irreducible with weight profile $(k_1,\ldots, k_d)$ and weight multiplicities $(m_1,\ldots, m_r)$, then the following identities hold:
\begin{enumerate}
\item if $j \geq 1$ and $1 \leq i \leq m_{j}$ then $k_{m_1+\cdots + m_{j-1}+i} = k_1 + 2j-2$,
\item $d = \sum_{i =1}^r m_i$,
\item $r = 1 + \frac{k_d-k_1}{2}$.
\end{enumerate}
\end{rmk}

\begin{lem}\label{lemmaunitarywts} If $\rho$ is an irreducible \emph{unitary} representation of $SL_2(\ZZ)$ distinct from the 1-dimensional trivial representation, then the weights $k_1, \hdots, k_d$ in the weight profile of $\rho$ lie in the range $[1, 11]$.
\end{lem}
\begin{proof} This is proved in Section 6 of \cite{CandeloriFranc}. We give a second proof: it is proved in Section 3 of \cite{KM1} (with further details in  Section 7 of \cite{KM2}) that  the classical Hecke estimate
$O(n^k)$ continues to hold for the $n^{th}$ Fourier coefficient of any component of a holomorphic vvmf of weight $k$ associated to a unitary representation $\rho$.\ Then a standard argument shows that if $\rho$ is irreducible and nontrivial, the weight of
a nonzero holomorphic vvmf is necessarily \emph{positive}.\ Hence, $k_1\geq 1$.

On the other hand, because $\rho$ is unitary then so is $\rho^{\vee}$.\ By Remark \ref{r:duality}, the lowest weight
in the cuspidal weight profile for $\rho^{\vee}$ is $12-k_d$, and by the argument of the previous paragraph we
have $12-k_d\geq 1$.
\end{proof}

\begin{dfn} Let $\rho$ be a representation of $SL_2(\ZZ)$.\ We say that $M(\rho)$ is \emph{cyclic} if it is a cyclic $R$-module, i.e.\ there is some weight $k$ vvmf $F$ such that $M(\rho)=R.F$.\ In this situation, we also say that $\rho$ itself is cyclic.
\end{dfn}
\begin{lem}\label{lemmacyclic} 
Suppose that $\rho$ is irreducible and that $\pi_{\rho}=(\overbrace{1, \hdots, 1}^t, \dots)$, i.e.\ there is $t\geq 1$ such that
the first $t$ weight multiplicities are $1$. Let $F$ be a nonzero vvmf of minimal weight $k_1$.\ Then either $\{F, DF, \hdots, D^{t-1}F\}$ is a complete set of  generators, or there is a generating set that contains $\{F, DF, ..., D^{t}F\}$.
\end{lem}
\begin{proof} By assumption there is a unique generator of weight $k_1$ (up to scalars), so we can always include $F$ in a set of free generators. Suppose that $F, DF, \ldots,D^iF$ are in a free generating set, and that $i\leq t-1$.\ By hypothesis, the \emph{only} free generators with weight between $k_1$ and $k_1+2i-2$ are the $D^iF$ ($0\leq i<t$) (up to scalars). If we \emph{cannot} include $F, \hdots, D^iF, D^{i+1}F$ in a generating set, there must be an expression of the form $D^{i+1}F=\sum_j f_jG_j$ where $f_j\in M_j$ is a classical modular form of \emph{positive weight} and each $G_j$ is a free generator. Then $\wt(G_j)\leq k_1+2i-2$, whence each $G_j$ is equal to some $D^jF\ (j<i)$ (up to scalars). Now it follows that the $M$-span of $F, \hdots, D^iF$ is an $R$-module, and  by Lemma \ref{lemmaRrk} it follows that $i+1=d$. Thus $t\geq i+1 =d\geq t$, whence $d=t=i+1$. 
 
This shows that if $i<t-1$ then we can \emph{always} adjoin $D^{i+1}F$ to a set of  free generators $\{F, \hdots, D^iF\}$ to obtain a larger such set of free generators.\ Similarly, if $i=t-1$ then either we can similarly adjoin $D^tF$, or else $t=d$ and $\{F, \hdots, D^{t-1}F\}$ is already a complete set of free generators. The Lemma follows.
 \end{proof}

\begin{lem}\label{lemmacyclicequivs} Let $\rho$ be an irreducible representation of $SL_2(\ZZ)$ of dimension $d$.\ The following are equivalent:
\begin{enumerate}
\item[(a)] $M(\rho)$ is cyclic,
\item[(b)] There is $F\in M_{k_1}(\rho)$ such that $\{F, DF, \hdots D^{d-1}F\}$ freely generates of $M(\rho)$,
\item[(c)] $\pi_{\rho} = (1, 1,\hdots, 1)$.
\end{enumerate}
\end{lem}
\begin{proof} If (a) holds, there is a vvmf $F$ of weight $k$ such that $M(\rho)=R.F$.\ Then  $M(\rho)=\sum_{i\geq 0} MD^iF \supseteq \sum_{i=0}^{d-1} MD^iF$, and by Lemma \ref{lemmacyclic} the last containment is an equality.\ Now (b) is a consequence of Lemma \ref{lemmaDFs}(a), and this shows that (a)$\Leftrightarrow$(b).\ Clearly
(b)$\Rightarrow$(c), while the converse also follows from Lemma \ref{lemmacyclic}.\ So (b)$\Leftrightarrow$(c),
and the proof of the Lemma is complete. 
\end{proof}

\begin{rmk} For further discussion of the case of cyclic $\rho$, see Theorem 1.3 and Section 4 of \cite{MM}.
\end{rmk}

\begin{ex} (a) For all $n\geq 0$, the $n^{th}$ symmetric power $S^n(\rho)$ of the defining 2-dimensional representation
$\rho$ of $SL_2(\ZZ)$ is irreducible and cyclic.\ These examples are discussed at length in \cite{KM3}.\\
(b) \emph{Every} irreducible $\rho$ of dimension $\leq 3$ is cyclic.\ See \cite{FM1} for an extensive discussion of these cases.
\end{ex}

In spite of these examples, it appears that there are not too many classes of irreducible $\rho$ which are cyclic, and it is an interesting  problem to try and \emph{classify} all examples.\ The unitary case seems particularly tractable, because of the next result.

\begin{lem}\label{p:cyclic}
Let $\rho$ be an irreducible, cyclic, unitary representation of $\SL_2(\ZZ)$.\ Then $\dim\rho\leq 6$.
\end{lem}
\begin{proof} 
We know from Lemma \ref{lemmacyclicequivs} that all weight multiplicities are equal to $1$ because $\rho$ is cyclic.\ On the other hand,
by Lemma \ref{lemmaunitarywts} there are no more than $5$ distinct weights thanks to unitarity.\ The only way to reconcile these statements is if the dimension $\dim\rho\leq 6$.  
\end{proof}

This section concludes with a result (Theorem \ref{l:goodbasis}) that will be used to prove upper bounds on weight multiplicities for irreducible representations (Theorem \ref{t:multiplicitybounds}). We begin with some preparations.

Let $F_1,\ldots, F_d$ denote a free basis for $M(\rho)$, chosen so that each $F_j$ has integer weight. Let $A = (a_{ij})$ denote the matrix of $D$ in this basis, so that $DF_j = \sum_{i = 1}^da_{ij}F_i$. If $F$ denotes the $d\times d$ matrix whose columns are the $F_j$, then $A$ is defined by the matrix equation $DF = FA$. If $F$ is replaced by $FP$ for some invertible matrix $P$ with entries in $M = \CC[E_4,E_6]$, and  if $A'$ is the matrix of $D$ with respect to this new basis, then
\[
  FPA' = D(FP) = D(F)P + FD(P) = FAP + FD(P)
\]
and thus $A' = P^{-1}AP + P^{-1}D(P)$. 

Suppose that $P$ corresponds to replacing a basis vector $F_j$ by $F_j-gF_i$ where $i < j$ and $g \in M$. We call this an \emph{elementary replacement operation}. The matrix of $D$ changes under such an elementary replacement operation as follows:
\begin{enumerate}
\item add $g$ times the $j$th row of $A$ to the $i$th row of $A$ and
\item subtract $g$ times the $i$th column of $A$ from the $j$th column of $A$ and
\item subtract $D(g)$ from the $(i,j)$-entry of the result.
\end{enumerate}
We will use elementary replacement operations to find a basis in which the matrix of $D$ has a particularly simple form. The idea will be to methodically winnow away copies of $E_6$. To this end, if $f \in M$, then let $d(f)$ denote the $E_6$-degree of $f$ when it is regarded as an element of the polynomial ring $\CC[E_4,E_6]$. For each integer $t \geq 0$, let
\begin{eqnarray*}
  M_k^t = \{f \in M_k \mid d(f) \leq t\}.
\end{eqnarray*}

\begin{thm}
\label{l:goodbasis}
Let $\rho$ denote an irreducible representation of $\SL_2(\ZZ)$ and let $L$ denote a choice of exponents for $\rho(T)$.\ Then there exists a basis for $M_L(\rho)$ consisting of integer weight vector valued modular forms such that the matrix of $D$ in this basis contains only entries that are multiples of pure monomials of the form $E_4^x$.
\end{thm}
\begin{proof}
Let $(k_1,\ldots, k_d)$ be the $L$-adapted weight profile of $\rho$ and let $r = 1 + \frac{k_d-k_1}{2}$. Let $m_1,\ldots, m_r$ denote the weight multiplicities. Choose free generators $F_j$ for $M(\rho)$ ordered by increasing weight. Hence $F_1,\ldots, F_{m_1}$ are of weight $k_1$,  $F_{m_1+1},\ldots, F_{m_1+m_2}$ are of weight $k_2 = k_1+2$, and so on. Define the matrix $A = (a_{ij})$ of $D$ in this basis by writing $D(F_j) = \sum_{i =1}^{d} a_{ij}F_i$ for all $j$. Then $A$ has the following block shape:
\[
\begin{array}{c|c|c|c|c|c|c|c|c|c|c|}
&m_1&m_2&m_3&m_4&m_5&m_6&m_7&\cdots&m_{r-1}&m_r\\
\hline
m_1&\star &4&6&8&10&12&14&&2r&2r+2\\
\hline
m_2&0&\star&4&6&8&10&12&&2r-2&2r\\
\hline
m_3&\star&0&\star&4&6&8&10&&2r-4&2r-2\\
\hline
m_4&\star&\star&0&\star&4&6&8&&2r-6&2r-4\\
\hline
m_5&\star&\star&\star&0&\star&4&6&&2r-8&2r-6\\
\hline
m_6&\star&\star&\star&\star&0&\star&4&&2r-10&2r-8\\
\hline
m_7&\star&\star&\star&\star&\star&0&\star&&2r-12&2r-10\\
\hline
\vdots&&&&&&&&&&\\
\hline
m_{r-1}&\star&\star&\star&\star&\star&\star&\star&&\star&4\\
\hline
m_r&\star&\star&\star&\star&\star&\star&\star&&0&\star\\
\hline
\end{array}
\]
The $\star$ entries indicate zeros, the integer entries indicate weights of the entries, and the row and column labels $m_j$ indicate the size of the blocks in the block matrix decomposition.

Notice that the weight of diagonal entries of $A$ is constant. We will slowly improve $A$ diagonal by diagonal using elementary replacement operations. Our goal is to use a sequence of elementary replacement operations to find a basis for $M(\rho)$ such that the matrix of $D$ in this basis has entries in $M_k^0$. It will be convenient, to phrase things in a uniform way, to regard $D$ as a matrix with infinitely many columns moving to the right and infinitely many rows moving down. Thus, initially the matrix of $D$ has the form
\[
\begin{array}{c|c|c|c|c|c|c|c|c|c|c|c|c|c}
&m_1&m_2&m_3&m_4&m_5&m_6&m_7&m_8&m_9&m_{10}&m_{11}&m_{12}&\\
\hline
m_1&\star &M_4^0&M_6^1&M_8^0&M_{10}^1&M_{12}^2&M_{14}^1&M_{16}^2&M_{18}^3&M_{20}^2&M_{22}^3&M_{24}^4&\\
\hline
m_2&M_0^0&\star&M_4^0&M_6^1&M_8^0&M_{10}^1&M_{12}^2&M_{14}^1&M_{16}^2&M_{18}^3&M_{20}^2&M_{22}^3&\\
\hline
m_3&\star&M_0^0&\star&M_4^0&M_6^1&M_8^0&M_{10}^1&M_{12}^2&M_{14}^1&M_{16}^2&M_{18}^3&M_{20}^2&\\
\hline
m_4&\star&\star&M_0^0&\star&M_4^0&M_6^1&M_8^0&M_{10}^1&M_{12}^2&M_{14}^1&M_{16}^2&M_{18}^3&\\
\hline
m_5&\star&\star&\star&M_0^0&\star&M_4^0&M_6^1&M_8^0&M_{10}^1&M_{12}^2&M_{14}^1&M_{16}^2&\\
\hline
m_6&\star&\star&\star&\star&M_0^0&\star&M_4^0&M_6^1&M_8^0&M_{10}^1&M_{12}^2&M_{14}^1&\cdots\\
\hline
m_7&\star&\star&\star&\star&\star&M_0^0&\star&M_4^0&M_6^1&M_8^0&M_{10}^1&M_{12}^2&\\
\hline
m_8&\star&\star&\star&\star&\star&\star&M_0^0&\star&M_4^0&M_6^1&M_8^0&M_{10}^1&\\
\hline
m_9&\star&\star&\star&\star&\star&\star&\star&M_0^0&\star&M_4^0&M_6^1&M_8^0&\\
\hline
m_{10}&\star&\star&\star&\star&\star&\star&\star&\star&M_0^0&\star&M_4^0&M_6^1&\\
\hline
m_{11}&\star&\star&\star&\star&\star&\star&\star&\star&\star&M_0^0&\star&M_4^0&\\
\hline
m_{12}&\star&\star&\star&\star&\star&\star&\star&\star&\star&\star&M_0^0&\star&\\
\hline
&&&&&&\vdots&&&&&&&\\
\end{array}
\]
where the $m_j$ labels indicate a block of rows or columns of size $m_j$, a block with an entry of the form $M_k^t$ means that the block matrix contains entries in $M_k^t$, and a $\star$ indicates that weight considerations force the entries in those blocks to be zero.

Our algorithm proceeds by using elementary replacement operations to change block diagonals with entries in $M_k^{t}$ to have entries in $M_k^{t-2}$, but one must take care in how one chooses the diagonals. The rule for choosing which block diagonal to adjust is to start looking from the center diagonal of zeros, and move up until you encouter a pair of adjacent diagonals containing entries in $M_{2+2k}^{t+1}$ and $M_{4+2k}^t$. Then, adjust the $k$th diagonal up from the center, which contains entries in $M_{2+2k}^{t+1}$. Afterward it will contain entries in $M_{2+2k}^{t-1}$, and then the algorithm repeats. This alogrithm will involve some backtracking, and so we must argue that it is possible to do such backtracking without undoing the operations that preceded it.

Let us explain the first step of the algorithm very carefully. Consider one of the $m_i \times m_{i+2}$ block matrices, which contains entries in $M_6^1 = M_6 = \langle E_6\rangle$. Since $D(E_4) = -\frac{1}{3}E_6$, we can replace basis vectors $F$ corresponding with the $(i+2)$th block column of $D$ with basis vectors of the form $F-\alpha E_4G$ for $\alpha \in \CC$ and $G$ some basis vector corresponding with the $m_1$ block of columns, and appropriate choices of $\alpha$ will allow us to ensure that all entries on this block diagonal are in $M_6^0 = 0$. Note that these elementary replacement operations will also affect the $m_{i}\times m_{i+1}$ and $m_{i+1}\times m_{i+2}$ blocks, but it will affect them by adding multiples of $E_4$ to entries. Thus, the result will still lie in $M_4^0 = M_4$. These operations will also affect block diagonals above the weight $6$ diagonal, but we don't care about that at this stage, as we haven't yet performed any simplifications to that part of the matrix. Thus, after all these replacements we reduce to a matrix for $D$ of the form
\[
\begin{array}{c|c|c|c|c|c|c|c|c|c|c|c|c|c}
&m_1&m_2&m_3&m_4&m_5&m_6&m_7&m_8&m_9&m_{10}&m_{11}&m_{12}&\\
\hline
m_1&\star &M_4^0&M_6^0&M_8^0&M_{10}^1&M_{12}^2&M_{14}^1&M_{16}^2&M_{18}^3&M_{20}^2&M_{22}^3&M_{24}^4&\\
\hline
m_2&M_0^0&\star&M_4^0&M_6^0&M_8^0&M_{10}^1&M_{12}^2&M_{14}^1&M_{16}^2&M_{18}^3&M_{20}^2&M_{22}^3&\\
\hline
m_3&\star&M_0^0&\star&M_4^0&M_6^0&M_8^0&M_{10}^1&M_{12}^2&M_{14}^1&M_{16}^2&M_{18}^3&M_{20}^2&\\
\hline
m_4&\star&\star&M_0^0&\star&M_4^0&M_6^0&M_8^0&M_{10}^1&M_{12}^2&M_{14}^1&M_{16}^2&M_{18}^3&\\
\hline
m_5&\star&\star&\star&M_0^0&\star&M_4^0&M_6^0&M_8^0&M_{10}^1&M_{12}^2&M_{14}^1&M_{16}^2&\\
\hline
m_6&\star&\star&\star&\star&M_0^0&\star&M_4^0&M_6^0&M_8^0&M_{10}^1&M_{12}^2&M_{14}^1&\cdots\\
\hline
m_7&\star&\star&\star&\star&\star&M_0^0&\star&M_4^0&M_6^0&M_8^0&M_{10}^1&M_{12}^2&\\
\hline
m_8&\star&\star&\star&\star&\star&\star&M_0^0&\star&M_4^0&M_6^0&M_8^0&M_{10}^1&\\
\hline
m_9&\star&\star&\star&\star&\star&\star&\star&M_0^0&\star&M_4^0&M_6^0&M_8^0&\\
\hline
m_{10}&\star&\star&\star&\star&\star&\star&\star&\star&M_0^0&\star&M_4^0&M_6^0&\\
\hline
m_{11}&\star&\star&\star&\star&\star&\star&\star&\star&\star&M_0^0&\star&M_4^0&\\
\hline
m_{12}&\star&\star&\star&\star&\star&\star&\star&\star&\star&\star&M_0^0&\star&\\
\hline
&&&&&&\vdots&&&&&&&\\
\end{array}
\]

Suppose now by induction that we've found a basis for $M(\rho)$ such that the diagonals of $D$ have entries in the following spaces:
\[
  M_4^0, M_6^0,\ldots,M_{4t}^0, M_{4t+2}^1,M_{4t+4}^2,\ldots, M_{6t-2}^{t-1},M_{6t}^t,M_{6t+2}^{t-1},M_{6t+4}^{t},M_{6t+6}^{t+1},M_{6t+8}^{t},\ldots
\]
where $t \geq 1$. Note that we have not put any restrictions on the $E_6$-degree of entries in the weight $6t$ diagonals and higher. Write entries  $f \in M_{6t}^t$ uniquely in the form $f=\alpha E_6^{t} + g$ for $\alpha \in \CC$ and $g \in M_{6t}^{t-1}$. Then $h = -\frac{\alpha}{2} E_4E_6^{t-1} \in M_{6t-2}^{t-1}$ satisfies $f - D(h) \in M_{6t}^{t-2}$. Thus, if we use such forms $h$ to perform elementary replacement operations, we can force the weight $6t$ diagonal to lie in $M_{6t}^{t-2}$. This will adjust the entries in the weight $6t-2$ diagonal by the various $h$'s that arise, but since these all lie in $M_{6t-2}^{t-1}$, we will not disrupt this diagonal. Similarly, these elementary replacement operations will alter diagonals above the weight $6t$ diagonal, but since we have not put any restrictions on those diagonals yet, such operations are inconsequential for our goal.

Now comes the slightly delicate part: we continue working backwards from the weight $6t-2$ block diagonal to the weight $2t+2$ block diagonal, and the issue is that we've adjusted diagonals from the one under consideration up to the weight $6t$ diagonal. The saving grace is that there are enough diagonals in low weights that do not contain any copies of $E_6$. 

More precisely, suppose that we've reduced to a matrix with diagonals of the form
\[
  M_4^0,\ldots,M_{4t}^0, M_{4t+2}^1,M_{4t+4}^2,\ldots,M_{6t-2j-2}^{t-j+1},M_{6t-2j}^{t-j},M_{6t-2j+2}^{t-j+1},\ldots, M_{6t-4}^{t-2},M_{6t-2}^{t-1},M_{6t}^{t-2},\ldots
\]
When we adjust the weight $6t-2j-2$ diagonal we must be careful not to disrupt the diagonals of weight $6t-2j$ through weight $6t$, since we have reduced the $E_6$ degree of each. However, we can ignore diagonals above this, as we have not put any restrictions on them yet. The elementary replacement operations that we perform in weight $6t-2j-2$ will involve multiples of $h = E_4^xE_6^{t-j} \in M_{6t-2j-4}^{t-j}$. Entries in the weight $6t-2j+2r$ diagonals, for $r = 0,\ldots, j$, will be adjusted by forms in $hM_{2r-4}^{0}$. The $E_6$-degree of zero arises since we have already ensured that the weight $4$ through $4t$ diagonals have $E_6$-degree equal to $0$, and since $j \leq t$, these are the diagonals that affect the diagonals that we're worried about when we perform the elementary replacement operations. Since $hM_{2r-4}^{0} \subseteq M_{6t-2j+2r}^{t-j} \subseteq M_{6t-2j+2r}^{t-j+r}$, we will not undo any of the hard work that we have done between weights $6t-2j-2$ and $6t$. As $j$ increases to $t$, we wind up with a matrix whose sequence of diagonals looks like
\[
  M_4^0,\ldots,M_{4t}^0, M_{4t+2}^0,M_{4t+4}^0,M_{4t+6}^1\ldots, M_{6t-2}^{t-3},M_{6t}^{t-2},M_{6t+2}^{t-1},M_{6t+4}^{t},M_{6t+6}^{t+1},M_{6t+8}^{t},\ldots
\]
And now we can repeat with the weight $6t+6$ diagonal. This proves the Theorem.
\end{proof}

\section{Bounds for weight multiplicities}
\label{s:multiplicitybounds}
\begin{thm}
\label{t:multiplicitybounds}
Let $\rho$ be an irreducible representation of $\SL_2(\ZZ)$ of dimension $\dim \rho \geq 2$. Let $\pi_{\rho}=(m_1,\ldots, m_r)$ denote the weight multiplicity tuple of $\rho$, and define $m_j = 0$ if $j < 1$ or if $j > r$. Then for all $j \geq 1$ we have
\[
  m_j \leq \sum_{t\geq 0}m_{j+1-2t}
\]
and
\[
  m_j \leq \sum_{t \geq 0}m_{j-1+2t}.
\]
In particular, $m_j \leq \frac{1}{2}\dim\rho$ for all $j$.
\end{thm}
\begin{proof}
We first explain how to establish the first inequality more generally for $M_L(\rho)$ for any choice of exponents $L$ for $\rho(T)$. Choose a basis for $M_L(\rho)$ as in Theorem \ref{l:goodbasis}, and assume to the contrary that there exists $j$ such that $m_j > \sum_{t \geq 0} m_{j+1-2t}$. Consider the matrix $A$ obtained from the $j$th block column of the matrix of $D$ in the chosen basis, but where we ignore the blocks that are known to be zero by weight considerations. This is a matrix with $\sum_{t\geq 0}m_{j+1-2t}$ rows and $m_j$ columns. Thus, by hypothesis $A$ has a nontrivial kernel consisting of a scalar vector. If $b = (b_v)$ is such a column vector, and if $F_1,\ldots, F_{m_j}$ denote the basis vectors of weight corresponding to the multiplicity $m_j$, then $F = \sum_{v=1}^{m_j} b_{v}F_{v}$ is nonzero and $D(F) = 0$, contradicting the irreducibility of $\rho$ since $\dim \rho \geq 2$ (Lemma \ref{lemmaRrk}).

The second inequality can be deduced from the first by duality, since $\cVbar_k(\rho)^\vee \cong \cSbar_{12-k}(\rho^\vee)$, and since Theorem \ref{l:goodbasis} and Lemma \ref{lemmaRrk} are valid for any choice of exponents.
\end{proof}

\begin{rmk}
 Computational evidence suggests that the stronger three-term inequality $m_{j} \leq m_{j+1}+m_{j-1}$ might hold. This would follow if one could find a basis for $M(\rho)$ such that the matrix of $D$ contains only constants and constant multiples of $E_4$. We were unable to prove this stronger result, save for under two different hypotheses:
\begin{enumerate}
\item If $\rho$ is irreducible and unitarizable, then it's known (see Section 6 of \cite{CandeloriFranc} or that the weight profile consists only or Lemma \ref{p:cyclic} above) that there are at most six multiplicities for $\rho$. In this case the two inequalities of Theorem \ref{t:multiplicitybounds} yield the three-term inequality $m_j \leq m_{j+1}+m_{j-1}$ for $j = 1,\ldots, 6$.

\item If $\rho$ is an irreducible representation and $\sigma$ is the standard representation of $\SL_2(\ZZ)$, then it's easy to relate the weight profiles of $\rho$ and $\rho \otimes \sigma$. Since $\sigma(T)$ has all exponents equal to zero, one has $\cVbar_k(\rho\otimes \sigma) = \cVbar_k(\rho)\otimes \cVbar_0(\sigma)$. In particular, since $\cVbar_0(\sigma) = \cO(1)\oplus \cO(-1)$, if $\cVbar_0(\rho) = \bigoplus_{r=1}^d\cO(-k_r)$, then
\[
  \cVbar_0(\rho\otimes \sigma) = \bigoplus_{r=1}^d\cO(-k_r+1)\oplus \cO(-k_r-1)
\]
Let $m_1,\ldots m_t$ be the multiplicities for $\rho$. Then the multiplicities for $\rho\otimes \sigma$ are $m_1,m_1+m_2,\ldots, m_{t-1}+m_t,m_t$. Thus the three-term inequality for $\rho \otimes\sigma$ boils down to $0\leq m_{j+1}+m_{j-2}$, which is trivially satisfied. This is true regardless of whether the three-term inequality was satisfied by the weight multiplicities of $\rho$.
\end{enumerate}
\end{rmk}

\begin{rmk}
It is worth remarking that in the case of the standard representation $\sigma$ of $\SL_2(\ZZ)$, one has
\begin{equation}
\label{eq:stdrep}
  \cVbar_0(\sigma) \cong \cO(-1)\oplus \cO(1).
\end{equation}
This reflects the fact that $\cVbar_0(\sigma)$ can be identified with the relative homology of the universal elliptic curve over the moduli stack of generalized elliptic curves. A vector valued modular form of minimal weight $-1$ for $\sigma$ is given by 
\[
 F(\tau) = \twovec{2\pi i \tau}{2\pi i}.
\]
The decomposition (\ref{eq:stdrep}) is the Hodge decomposition for the relative homology of the universal elliptic curve, and one might ask to what extent such a relationship holds for other representations of $\SL_2(\ZZ)$. 
\end{rmk}

We end this Section by looking more closely at the bound $m_j\leq d/2$ for weight multiplicities given in Theorem \ref{t:multiplicitybounds}, where $d=\dim\rho$.\ We will show (Lemma \ref{lemmaellebound}) that if $\ell$ is the number of \emph{distinct} weight multiplicities 
and $e$ the \emph{minimum} of the (nonnegative) integers $[d/2]-m_j\ (j\geq 1)$, then $\ell/e\leq 8$.\ We can be more precise
for small $e$.\ First we treat the case $e=0$, where we show that $\ell\leq 3$.
\begin{lem}
\label{lemma2wts} 
Suppose that $\rho$ is irreducible.\ There are exactly 2 distinct weight multiplicities in the weight profile of $\rho$ if, and only if, $\dim\rho =2$.
\end{lem}
\begin{proof} The result is clear if $\dim\rho=2$, so assume that $m_1, m_2$ are the two weight multiplicities.\
Then we must have $m_1=m_2=d/2$, because neither multiplicity may exceed $d/2$.\  Let $F_1, \hdots, F_{d/2}$ be a basis for the vvmfs of least weight $k_0$.\ Then it is easy to see that
that $DF_1, \hdots, DF_{d/2}$ may be chosen as the free generators of weight $k_1=k_0+2$, so that we have
relations of the form $D^2F_j = E_4\sum_{j=1}^{d/2}a_{ij}F_i\ \ (a_{ij}\in\CC)$.

Let $\lambda$ be an eigenvalue of the matrix of coefficients $(a_{ij})$ corresponding to a nonzero
$F$ in the linear span of the $F_i$s.\ Then we have $D^2F=\lambda E_4F$, so that $F$ satisfies an
order 2 MLDE. Therefore $\dim\rho=2$ because $\rho$ is irreducible.
\end{proof}

\begin{lem}\label{lemmad/2mult} Suppose that $\rho$ is irreducible and some weight multiplicity in the weight profile of $\rho$ is $d/2$.\ Then either $\dim\rho=2$, or the multiplicity profile has the form $(m_1, d/2, m_3)$ (and in particular, there are just $3$ weights).
\end{lem}
\begin{proof} Let $m_j=d/2$.\ By the inequality of Theorem \ref{t:multiplicitybounds} we have
\begin{eqnarray*}
d/2\leq \sum_{t\geq -1} m_{j-1-2t}\leq d/2.
\end{eqnarray*}
Therefore, \emph{all} nonzero weight multiplicities already appear in the displayed inequalities.\ By the no-gap Lemma,
there must be  either 2 or 3 nonzero multiplicities, and if there are 2 then $\dim\rho = 2$ by
Lemma \ref{lemma2wts}.\ If there are 3 then we cannot have $m_1=d/2$ because $m_1\leq m_2$,
and similarly $m_3=d/2$ is ruled out.\ Therefore $m_2=d/2$, and the Lemma is proved.
\end{proof}

It is evident that the argument of the last Lemma can be systematized.\ The general idea is that
the inequality of Theorem \ref{t:multiplicitybounds} involves mainly multiplicities $m_{j-1-2t}$ (the point being that the subscripts have the same parity), whereas the no-gap Lemma says that there must also
be (nonzero) multiplicities for the intermediate multiplicities $m_{j'}$ with $j'\equiv j$ (mod $2$).\ In the general case we can argue as follows.\
For each weight multiplicity $m_j$, define $e_j:= [d/2]-m_j\geq 0$.\ By Theorem
\ref{t:multiplicitybounds} we have
\begin{eqnarray*}
m_j+ \sum_{t\geq -1} m_{j-1-2t}\geq 2m_j = 2[d/2]-2e_j.
\end{eqnarray*}
The number of integers $j'$ in the range $[1, j-2]$ satisfying  $j'\equiv j$ (mod $2$)
is $[j-1/2]$.\ By the no-gap Lemma
we have $m_{j'}\geq 1$ for these $j'$, whence we obtain
\begin{eqnarray*}
d = \sum_j m_j \geq [j-1/2]+2[d/2]-2e_j.
\end{eqnarray*}
This implies that
\begin{eqnarray*}
j\leq 4(e_j+1).
\end{eqnarray*}
In a nutshell, if we have a multiplicity $m_j$ that is `not too far' from $d/2$
(i.e., $e_j$ is small) then $j$ must be small too.\ For example, if
some $m_j=d/2\Rightarrow e_j=0\Rightarrow j\leq 1$ (because $d$ is even), and we easily recover the results of Lemma \ref{lemmad/2mult} in this case.

Let $e\df \min_j\ {e_j}$ be as before, with $e = e_{j_0}$.\ There may be several such $j_0$, but they all satisfy
 $j_0\leq 4(e+1)$.\ By duality, all of these arguments
apply to the cuspidal weight profiles too, and we know that in these cases the weight multiplicities are
reversed upon passing from $\rho$ to $\rho^{\vee}$ (cf. Remark \ref{r:duality}).\ Moreover, $e$ is the same for
$\rho$ and $\rho^{\vee}$. Therefore, not only must 
the minimum discrepancy $e$ occur by the time we reach the $4(e+1)^{th}$ weight multiplicity, the last time
the minimum discrepancy occurs must be
within the same distance of the highest weight.\ Therefore, as there are exactly $\ell$ distinct weight multiplicities,
then $\ell \leq 8e+7$.\ We state this as
\begin{lem}\label{lemmaellebound} Let $\rho$ be irreducible and suppose that $\pi_{\rho}=(m_1, \hdots, m_{\ell})$.\ Let $e$ be the minimum value of $[d/2]-m_j\ (1\leq j \leq \ell)$.\ Then
\begin{eqnarray*}
\ell \leq 8e+7.
\end{eqnarray*}
$\hfill \Box$
\end{lem}

\section{Multiplicity tables in low dimensions}
\label{s:examples}

The following lists of multiplicity profiles $\pi_\rho$ for irreducible representations $\rho$ of $\SL_2(\ZZ)$ were generated by a computer using the results discussed in Section \ref{s:weightprofiles}, the no-gap lemma (Lemma \ref{lemmanogap}), and Theorem \ref{t:multiplicitybounds}.\ They contain all multiplicity profiles that can arise from irreducible representations in dimensions six through ten, but our lists may include some examples that do not occur in practice\footnote{Since we do not have explicit equations for the character variety of $\SL_2(\ZZ)$ in dimensions six or greater, we do not know that there in fact exist representations $\rho$ of $\SL_2(\ZZ)$ having all of the possible prescribed values for $\Tr(\rho(R))$, $\Tr(\rho(S))$ and $\Tr(L)$ satisfying the obvious constraints.}. In each dimension there is a unique multiplicity tuple of length $d$, all of whose entries are one.\ This corresponds to the case of cyclic $\rho$ (Lemma \ref{lemmacyclicequivs}).\  Similarly, in dimension $d \geq 4$ there are $d-3$ tuples of length $d-2$, all entries of which are one save for a single two (which cannot occur in the first or last entries). We omit these from our lists in dimension seven and higher in order to fit the data within the margins.

\subsection{$d = 6$}
Total number of types: $\leq 10$.

\begin{center}
\begin{tabular}{|c|c|c|c|}
$[m_1,\ldots ,m_6]$ & $[m_1,\ldots,m_5]$&$[m_1,\ldots ,m_4]$ & $[m_1,m_2,m_3]$ \\
\hline
$\left[1, 1, 1, 1,1,1\right]$ &$\left[1, 1, 1, 2, 1\right]$&$\left[1, 1, 2, 2\right]$ & $\left[1, 3, 2\right]$ \\
&$\left[1,1, 2, 1, 1\right]$&$ \left[1, 2, 2, 1\right] $&$\left[2, 2, 2\right] $\\
&$\left[1,2, 1, 1, 1\right]$&$ \left[2, 2, 1, 1\right] $&$ \left[2, 3, 1\right] $ 
\end{tabular}
\end{center}

\subsection{$d = 7$}
Total number of types: $\leq 19$. 

\begin{center}
\begin{tabular}{|c|c|c|}
$[m_1,\ldots,m_5]$&$[m_1,\ldots,m_4]$&$[m_1,m_2,m_3]$\\
\hline
$\left[1, 1, 1, 2, 2\right] $&$\left[1, 1, 3, 2\right] $&$ \left[1, 3, 3\right] $\\
$\left[1, 1, 2, 2, 1\right] $&$\left[1, 2, 2, 2\right] $&$ \left[2, 3, 2\right] $\\
$\left[1, 2, 1, 2, 1\right] $&$\left[1, 2, 3, 1\right] $&$ \left[3, 3, 1\right] $\\
$[1,2,2,1,1]$&$\left[1, 3, 2, 1\right]$&\\
$ \left[2, 2, 1, 1, 1\right] $&$ \left[2, 2, 2, 1\right] $&\\
&$ \left[2, 3, 1, 1\right] $&
\end{tabular}
\end{center}

\subsection{$d = 8$}
Total number of types: $\leq 38$

\begin{center}
\begin{tabular}{|c|c|c|c|}
$[m_1,\ldots, m_6]$&$[m_1,\ldots, m_5]$&$[m_1,\ldots, m_4]$&$[m_1,m_2, m_3]$\\
\hline
$\left[1, 1, 1, 1, 2, 2\right]$&$\left[1, 1, 1, 3, 2\right]$&$\left[1, 1, 3, 3\right]$&$\left[1, 4, 3\right]$\\
$\left[1, 1, 1, 2, 2, 1\right]$&$\left[1, 1, 2, 2, 2\right]$&$\left[1, 2, 3, 2\right]$&$\left[2, 3, 3\right]$\\
$\left[1, 1, 2, 1, 2, 1\right]$&$\left[1, 1, 2, 3, 1\right]$&$\left[1, 3, 2, 2\right]$&$\left[2, 4, 2\right]$\\
$\left[1, 1, 2, 2, 1, 1\right]$&$\left[1, 1, 3, 2, 1\right]$&$\left[1, 3, 3, 1\right]$&$\left[3, 3, 2\right]$\\
$\left[1, 2, 1, 1, 2, 1\right]$&$\left[1, 2, 1, 2, 2\right]$&$\left[2, 2, 2, 2\right]$&$\left[3, 4, 1\right]$\\
$\left[1, 2, 1, 2, 1, 1\right]$&$\left[1, 2, 2, 2, 1\right]$&$\left[2, 2, 3, 1\right]$&\\
$\left[1, 2, 2, 1, 1, 1\right]$&$\left[1, 2, 3, 1, 1\right]$&$\left[2, 3, 2, 1\right]$&\\
$\left[2, 2, 1, 1, 1, 1\right]$&$\left[1, 3, 2, 1, 1\right]$&$\left[3, 3, 1, 1\right]$&\\
&$\left[2, 2, 1, 2, 1\right]$&&\\
&$\left[2, 2, 2, 1, 1\right]$&&\\
&$\left[2, 3, 1, 1, 1\right]$&&
\end{tabular}
\end{center}

\newpage

\subsection{$d = 9$}
Total number of types: $\leq 72$
\begin{center}
\begin{tabular}{|c|c|c|c|c|}
$[m_1,\ldots,m_7]$&$[m_1,\ldots,m_6]$&$[m_1,\ldots,m_5]$&$[m_1,\ldots,m_4]$&$[m_1,m_2,m_3]$\\
\hline
$\left[1, 1, 1, 1, 1, 2, 2\right]$&$\left[1, 1, 1, 1, 3, 2\right]$&$ \left[1, 1, 1, 3, 3\right] $&$\left[1, 1, 4, 3\right]$&$\left[1, 4, 4\right]$\\
$\left[1, 1, 1, 1, 1, 3, 1\right]$&$\left[1, 1, 1, 2, 2, 2\right]$&$\left[1, 1, 2, 3, 2\right]$&$\left[1, 2, 3, 3\right]$&$\left[2, 4, 3\right]$\\
$\left[1, 1, 1, 1, 2, 2, 1\right]$&$\left[1, 1, 1, 2, 3, 1\right]$&$\left[1, 1, 3, 2, 2\right]$&$\left[1, 2, 4, 2\right]$&$\left[3, 3, 3\right]$\\
$\left[1, 1, 1, 2, 1, 2, 1\right]$&$\left[1, 1, 1, 3, 2, 1\right]$&$\left[1, 1, 3, 3, 1\right]$&$\left[1, 3, 3, 2\right]$&$\left[3, 4, 2\right]$\\
$\left[1, 1, 1, 2, 2, 1, 1\right]$&$\left[1, 1, 2, 1, 2, 2\right]$&$\left[1, 2, 1, 3, 2\right]$&$\left[1, 3, 4, 1\right]$&$\left[4, 4, 1\right]$\\
$\left[1, 1, 2, 1, 1, 2, 1\right]$&$\left[1, 1, 2, 2, 2, 1\right]$&$\left[1, 2, 2, 2, 2\right]$&$\left[1, 4, 3, 1\right]$& \\
$\left[1, 1, 2, 1, 2, 1, 1\right]$&$\left[1, 1, 2, 3, 1, 1\right]$&$\left[1, 2, 2, 3, 1\right]$&$\left[2, 2, 3, 2\right]$& \\
$\left[1, 1, 2, 2, 1, 1, 1\right]$&$\left[1, 1, 3, 2, 1, 1\right]$&$\left[1, 2, 3, 2, 1\right]$&$\left[2, 3, 2, 2\right]$& \\
$\left[1, 2, 1, 1, 1, 2, 1\right]$&$\left[1, 2, 1, 1, 2, 2\right]$&$\left[1, 3, 2, 2, 1\right]$&$\left[2, 3, 3, 1\right]$& \\
$\left[1, 2, 1, 1, 2, 1, 1\right]$&$\left[1, 2, 1, 2, 2, 1\right]$&$\left[1, 3, 3, 1, 1\right]$&$\left[2, 4, 2, 1\right]$& \\
$\left[1, 2, 1, 2, 1, 1, 1\right]$&$\left[1, 2, 2, 1, 2, 1\right]$&$\left[2, 2, 1, 2, 2\right]$&$\left[3, 3, 2, 1\right]$& \\
$\left[1, 2, 2, 1, 1, 1, 1\right]$&$\left[1, 2, 2, 2, 1, 1\right]$&$\left[2, 2, 2, 2, 1\right]$&$\left[3, 4, 1, 1\right]$& \\
$\left[1, 3, 1, 1, 1, 1, 1\right]$&$\left[1, 2, 3, 1, 1, 1\right]$&$\left[2, 2, 3, 1, 1\right]$&& \\
$\left[2, 2, 1, 1, 1, 1, 1\right]$&$\left[1, 3, 2, 1, 1, 1\right]$&$\left[2, 3, 1, 2, 1\right]$&&\\
&$\left[2, 2, 1, 1, 2, 1\right]$&$\left[2, 3, 2, 1, 1\right]$&& \\
&$\left[2, 2, 1, 2, 1, 1\right]$&$\left[3, 3, 1, 1, 1\right]$&& \\
&$\left[2, 2, 2, 1, 1, 1\right]$&&& \\
&$\left[2, 3, 1, 1, 1, 1\right]$&&&
\end{tabular}
\end{center}

\newpage
\subsection{$d = 10$}
Total number of types: $\leq 142$
\begin{center}
\begin{tabular}{|c|c|c|c|c|c|}
$[m_1,\ldots,m_8]$&$[m_1,\ldots,m_7]$&$[m_1,\ldots,m_6]$&$[m_1,\ldots,m_5]$&$[m_1,\ldots,m_4]$&$[m_1,m_2,m_3]$\\
\hline
$\left[1, 1, 1, 1, 1, 1, 2, 2\right]$&$\left[1, 1, 1, 1, 1, 3, 2\right]$&$\left[1, 1, 1, 1, 3, 3\right]$&$\left[1, 1, 1, 4, 3\right]$&$\left[1, 1, 4, 4\right]$&$\left[1, 5, 4\right]$ \\
$\left[1, 1, 1, 1, 1, 2, 2, 1\right]$&$\left[1, 1, 1, 1, 2, 2, 2\right]$&$\left[1, 1, 1, 2, 3, 2\right]$&$\left[1, 1, 2, 3, 3\right]$&$\left[1, 2, 4, 3\right]$&$\left[2, 4, 4\right]$ \\
$\left[1, 1, 1, 1, 1, 3, 1, 1\right]$&$\left[1, 1, 1, 1, 2, 3, 1\right]$&$\left[1, 1, 1, 3, 2, 2\right]$&$\left[1, 1, 2, 4, 2\right]$&$\left[1, 3, 3, 3\right]$&$\left[2, 5, 3\right]$ \\
$\left[1, 1, 1, 1, 2, 1, 2, 1\right]$&$\left[1, 1, 1, 1, 3, 2, 1\right]$&$\left[1, 1, 1, 3, 3, 1\right]$&$\left[1, 1, 3, 3, 2\right]$&$\left[1, 3, 4, 2\right]$&$\left[3, 4, 3\right]$ \\
$\left[1, 1, 1, 1, 2, 2, 1, 1\right]$&$\left[1, 1, 1, 2, 1, 2, 2\right]$&$\left[1, 1, 2, 1, 3, 2\right]$&$\left[1, 1, 3, 4, 1\right]$&$\left[1, 4, 3, 2\right]$&$\left[3, 5, 2\right]$\\
$\left[1, 1, 1, 2, 1, 1, 2, 1\right]$&$\left[1, 1, 1, 2, 1, 3, 1\right]$&$\left[1, 1, 2, 2, 2, 2\right]$&$\left[1, 1, 4, 3, 1\right]$&$\left[1, 4, 4, 1\right]$&$\left[4, 4, 2\right]$ \\
$\left[1, 1, 1, 2, 1, 2, 1, 1\right]$&$\left[1, 1, 1, 2, 2, 2, 1\right]$&$\left[1, 1, 2, 2, 3, 1\right]$&$\left[1, 2, 1, 3, 3\right]$&$\left[2, 2, 3, 3\right]$&$\left[4, 5, 1\right]$ \\
$\left[1, 1, 1, 2, 2, 1, 1, 1\right]$&$\left[1, 1, 1, 2, 3, 1, 1\right]$&$\left[1, 1, 2, 3, 2, 1\right]$&$\left[1, 2, 2, 3, 2\right]$&$\left[2, 2, 4, 2\right]$& \\
$\left[1, 1, 2, 1, 1, 1, 2, 1\right]$&$\left[1, 1, 1, 3, 2, 1, 1\right]$&$\left[1, 1, 3, 2, 2, 1\right]$&$\left[1, 2, 3, 2, 2\right]$&$\left[2, 3, 3, 2\right]$& \\
$\left[1, 1, 2, 1, 1, 2, 1, 1\right]$&$\left[1, 1, 2, 1, 1, 2, 2\right]$&$\left[1, 1, 3, 3, 1, 1\right]$&$\left[1, 2, 3, 3, 1\right]$&$\left[2, 3, 4, 1\right]$&\\
$\left[1, 1, 2, 1, 2, 1, 1, 1\right]$&$\left[1, 1, 2, 1, 1, 3, 1\right]$&$\left[1, 2, 1, 1, 3, 2\right]$&$\left[1, 2, 4, 2, 1\right]$&$\left[2, 4, 2, 2\right]$&\\
$\left[1, 1, 2, 2, 1, 1, 1, 1\right]$&$\left[1, 1, 2, 1, 2, 2, 1\right]$&$\left[1, 2, 1, 2, 2, 2\right]$&$\left[1, 3, 2, 2, 2\right]$&$\left[2, 4, 3, 1\right]$&\\
$\left[1, 1, 3, 1, 1, 1, 1, 1\right]$&$\left[1, 1, 2, 2, 1, 2, 1\right]$&$\left[1, 2, 1, 2, 3, 1\right]$&$\left[1, 3, 2, 3, 1\right]$&$\left[3, 3, 2, 2\right]$&\\
$\left[1, 2, 1, 1, 1, 1, 2, 1\right]$&$\left[1, 1, 2, 2, 2, 1, 1\right]$&$\left[1, 2, 1, 3, 2, 1\right]$&$\left[1, 3, 3, 2, 1\right]$&$\left[3, 3, 3, 1\right]$& \\
$\left[1, 2, 1, 1, 1, 2, 1, 1\right]$&$\left[1, 1, 2, 3, 1, 1, 1\right]$&$\left[1, 2, 2, 1, 2, 2\right]$&$\left[1, 3, 4, 1, 1\right]$&$\left[3, 4, 2, 1\right]$&\\
$\left[1, 2, 1, 1, 2, 1, 1, 1\right]$&$\left[1, 1, 3, 2, 1, 1, 1\right]$&$\left[1, 2, 2, 2, 2, 1\right]$&$\left[1, 4, 3, 1, 1\right]$&$\left[4, 4, 1, 1\right]$&\\
$\left[1, 2, 1, 2, 1, 1, 1, 1\right]$&$\left[1, 2, 1, 1, 1, 2, 2\right]$&$\left[1, 2, 2, 3, 1, 1\right]$&$\left[2, 2, 1, 3, 2\right]$&& \\
$\left[1, 2, 2, 1, 1, 1, 1, 1\right]$&$\left[1, 2, 1, 1, 1, 3, 1\right]$&$\left[1, 2, 3, 1, 2, 1\right]$&$\left[2, 2, 2, 2, 2\right]$&& \\
$\left[2, 2, 1, 1, 1, 1, 1, 1\right]$&$\left[1, 2, 1, 1, 2, 2, 1\right]$&$\left[1, 2, 3, 2, 1, 1\right]$&$\left[2, 2, 2, 3, 1\right]$&& \\
&$\left[1, 2, 1, 2, 1, 2, 1\right]$&$\left[1, 3, 2, 2, 1, 1\right]$&$\left[2, 2, 3, 2, 1\right]$&&\\
&$\left[1, 2, 1, 2, 2, 1, 1\right]$&$\left[1, 3, 3, 1, 1, 1\right]$&$\left[2, 3, 1, 2, 2\right]$&&\\
&$\left[1, 2, 2, 1, 1, 2, 1\right]$&$\left[2, 2, 1, 1, 2, 2\right]$&$\left[2, 3, 2, 2, 1\right]$&&\\
&$\left[1, 2, 2, 1, 2, 1, 1\right]$&$\left[2, 2, 1, 2, 2, 1\right]$&$\left[2, 3, 3, 1, 1\right]$&&\\
&$\left[1, 2, 2, 2, 1, 1, 1\right]$&$\left[2, 2, 2, 1, 2, 1\right]$&$\left[2, 4, 2, 1, 1\right]$&&\\
&$\left[1, 2, 3, 1, 1, 1, 1\right]$&$\left[2, 2, 2, 2, 1, 1\right]$&$\left[3, 3, 1, 2, 1\right]$&&\\
&$\left[1, 3, 1, 1, 1, 2, 1\right]$&$\left[2, 2, 3, 1, 1, 1\right]$&$\left[3, 3, 2, 1, 1\right]$&&\\
&$\left[1, 3, 1, 1, 2, 1, 1\right]$&$\left[2, 3, 1, 1, 2, 1\right]$&$\left[3, 4, 1, 1, 1\right]$&&\\
&$\left[1, 3, 1, 2, 1, 1, 1\right]$&$\left[2, 3, 1, 2, 1, 1\right]$&&&\\
&$\left[1, 3, 2, 1, 1, 1, 1\right]$&$\left[2, 3, 2, 1, 1, 1\right]$&&&\\
&$\left[2, 2, 1, 1, 1, 2, 1\right]$&$\left[3, 3, 1, 1, 1, 1\right]$&&&\\
&$\left[2, 2, 1, 1, 2, 1, 1\right]$&&&&\\
&$\left[2, 2, 1, 2, 1, 1, 1\right]$&&&&\\
&$\left[2, 2, 2, 1, 1, 1, 1\right]$&&&&\\
&$\left[2, 3, 1, 1, 1, 1, 1\right]$&&&&
\end{tabular}
\end{center}

\bibliographystyle{plain}

\end{document}